\documentclass[a4paper,reqno]{amsart}
\usepackage{amsmath}
\usepackage{amsfonts}
\usepackage{amssymb}
\usepackage{xcolor}

\newtheorem{theorem}{Theorem}[section]
\newtheorem{lemma}[theorem]{Lemma}
\newtheorem{definition}[theorem]{Definition}

\newtheorem{corollary}[theorem]{Corollary}

\newcommand{\RR}{\mathbb{R}}

\newcommand{\h}{\mathcal{H}^{m}}

\newcommand{\I}{\mathcal{I}}
\newcommand{\M}{MMP}

\newcommand{\diam}{\text{diam}}

\title[Restricted families of projections]{Restricted families of projections and random subspaces}

\author{Changhao Chen}

\address{School of Mathematics and Statistics, The University of New South Wales, Sydney NSW 2052, Australia }
\email{changhao.chenm@gmail.com}

\date{\today}

\thanks{The author acknowledges the support of the Vilho, Yrj\"o, and Kalle V\"ais\"al\"a foundation.}

\begin{document}

\begin{abstract}
We study the restricted families of orthogonal projections in $\RR^{3}$. We show that there are  families of  random subspaces which admit a Marstrand-Mattila type projection theorem. 
\end{abstract}

\maketitle

\section{Introduction}

A fundamental problem in fractal geometry is to determine how the projections affect dimension. Recall the classical Marstrand-Mattila projection theorem: Let $E\subset \RR^{n}, n\geq2,$ be a Borel set with Hausdorff dimension $s$. 
\begin{itemize}
\item (dimension part) If $s\leq m$, then the orthogonal projection of $E$ onto almost all $m$-dimensional subspaces has Hausdorff dimension $s$.

\item (measure part) If $s>m$, then the orthogonal of $E$ onto almost all $m$-dimensional subspaces has positive $m$-dimensional Lebesgue measure. 
\end{itemize}

In 1954 J. Marstand \cite{Marstrand} proved this projection theorem
in the plane. In 1975 	P. Mattila \cite{Mattila1975} 
proved this  for general dimension via   1968 R. Kaufman's  \cite{Kaufman} potential theoretic methods.   We refer to the recent survey of P. Mattila \cite{Mattila2017},  K. Falconer, J. Fraser, and X. Jin \cite{Falconer}  for more backgrounds. For monographs which  are related to orthogonal projections of fractal sets, we refer to K. Falconer \cite{Falconer2003},    
P. Mattila \cite{Mattila1995}, \cite{Mattila2015}.

In this paper, we study the restricted families of projections in Euclidean spaces.
Let $G(n,m)$ denote the collections of all the $m$-dimensional linear subspaces of $\mathbb{R}^{n}$. For $V\in G(n,m)$, let $\pi_{V}: \RR^{n}\rightarrow V$ stand for the orthogonal projections onto $V$. For $G\subset G(n,m)$, we  call $(\pi_{V})_{V\in G}$ a restricted family of projections. One of the problems is to look for some ``strict'' subset  $G\subset G(n,m)$ such that the Marstrand-Mattila type theorem holds for this restricted families of projections $(\pi_{V})_{V\in G}$. 

The best possible lower bounds for general restricted families of projections $(\pi_{V})_{V\in G}$  (here $G$ is a smooth subset of $G(n,m)$) were obtained by E. J\"arvenp\"a\"a, M. J\"arvenp\"a\"a, T. Keleti, F. Ledrappier and M. Leikas, see \cite{JJK} and \cite{JJLL}. 

What kind of subset $G\subset G(n,m)$ admit a better lower bound or even more such that the Marstrand-Mattila type theorem holds?  K. F\"assler and T. Orponen \cite[Conjecture 1.6]{FOO} conjectured that if $G$ has ``curvature condition'' then $G$ admit a Marstrand-Mattila type theorem.  The following discription of T. Orponen \cite{OrponenH} is helpful.   ``Informally speaking, one could   conjecture that any (smooth) subset  $G\subset G(n,m)$ such that no ``large part'' of $G$ contained in a single non-trivial subspace, should satisfy the Marstrand-Mattila projection theorem''.   A prototypical example of a curve with curvature condition is given by
\[
\Gamma=\{\frac{1}{\sqrt{2}}(\cos \theta, \sin \theta,1), \theta \in  [0,2\pi)\}.  
\]
Recently, A. K\"aenm\"aki, T. Orponen, and L. Venieri \cite{KOV} proved that the dimensional part of Marstrand-Mattila type theorem holds for the restricted families of projections for the curve $\Gamma$, which partially answered a conjecture of \cite[Conjecture 1.6]{FOO} for the curve $\Gamma$.  We refer to \cite{KOV} for more details and references  therein. For the restricted families of projections $\{\pi_{V_{e}}\}_{e\in \Gamma}$ where $V_{e}:=e^{\perp}$ the orthogonal complement space of $e$, we refer to \cite{OV} for more details and new improvement. We note that D.  Oberlin and R. Oberlin \cite{OO} applied the Fourier restricted estimates to these  restricted families of projections with the ``curvature condition''.

We note that the subsets  $G$ which were mentioned in the former  results are always some smooth subsets of $G(n,m)$.  In a talk of T. Orponen, he talked about the restricted families of projections over general subsets of $G(n,m)$. Furthermore, he also considered the random subsets of $G(2,1)$ for another topic which is related to orthogonal projections. In this paper, inspired by T. Orponen's talk, we study the restricted families of projections over random subset of $G(n,m)$.  We show that there exist  non-smooth (fractal)  subsets of $G(3,1)$ such that the Marstrand-Mattila type theorem holds on this restricted family of projections.  Here the random sets play the same role as the sets with curvature condition.

We note that the random sets play the same role as curvature condition in some other situations also, e.g., the restricted Fourier transform, see T. Mitsis \cite{M} and G. Mockenhaupt \cite{GM}.

\begin{definition}[$\M$ spaces]
Let $G\subset G(n,m)$ and $\gamma$ be a nonzero finite Borel measure on $G$. We call the pair $((\pi_{V})_{V\in G}, \gamma)$  a $\M$ space if the  Marstrand-Mattila projection theorem holds for the restricted families of projections $(\pi_{V})_{V\in G}$ with respect to the measure $\gamma$.
\end{definition}


By ``mapping'' a class of  random Cantor sets of P. Shmerkin and V. Suomala \cite{SS} onto the sphere $\mathbb{S}^{2}$, and combing some classical potential arguments for orthogonal projections, we obtain the following Theorem \ref{thm:main}. 

Let $x\neq 0$. Denoted by  $L_{x}$ the line through zero and the point $x$, and $L_{x}^{\perp}$ the orthogonal complement of $L_{x}$.   
For convenience, we may identify with subset of $\mathbb{S}^{2}$ with subset of $G(3,1)$ which  makes no confusion.

\begin{theorem}\label{thm:main}
For any $1<\alpha \leq  2$ there exists an $\alpha$- Ahlfors regular set $G\subset \mathbb{S}^{2}$ such that $((\pi_{L_{x}})_{x\in G}, \mathcal{H}^{\alpha})$ and $((\pi_{L_{x}^{\perp}})_{x\in G},  \mathcal{H}^{\alpha})$  are $\M$ spaces. 
\end{theorem}

Recall that $E\subset \RR^{n}$ is called $\alpha$-Ahlfors regular for $0<\alpha\leq n$, if there exists a positive constant $C$ such that 
\[
r^{\alpha}/C\leq\mathcal{H}^{\alpha}(E\cap B(x,r))\leq Cr^{\alpha}
\]
for all $x\in E$ and $0 < r < \diam(E)$, where $\diam(E)$ denotes the diameter
of $E$. Note that for the case $\alpha=2$, Theorem \ref{thm:main} follows from the Marstrand-Mattila projection theorem. Thus we consider the case $\alpha \in (1, 2)$ only. 

I thank Tuomas Orponen for  pointing out that  the technique in the proof of Theorem \ref{thm:main} and  the random sets in papers \cite{Chend}, \cite{OOO} will imply the following result.

\begin{theorem}\label{thm:new}
 For any $0<\alpha \leq 1$ there exists a set $G\subset \mathbb{S}^{2}$ with $0<\mathcal{H}^{\alpha}(G)<\infty$ such that  $((\pi_{L_{x}})_{x\in G}, \mathcal{H}^{\alpha})$ admit a (dimension part)  Marstrand-Mattila type  theorem i.e., for any subset $E\subset \RR^{3}$ with $\dim_{H}E\leq \alpha$,
\begin{equation*}
\dim_{H}\pi_{L_{x}}(E)=\dim_{H}E \text{ for } \mathcal{H}^{\alpha}\,\, a.e.\, x\in G.
\end{equation*}
\end{theorem}

Recently there has been a growing interest in studying finite field version of some classical problems arising from Euclidean spaces. In \cite{ChenP}, the author  studied the  projections in vector spaces over finite fields, and obtained the Marstrand-Mattila type projection theorem in this setting.  For finite fields version of restricted families of projection, the author \cite{ChenRF} obtained that a random collection of subspaces admit a Marstrand-Mattila type theorem with high probability.  For more details  on finite fields version of projections, and  finite fields version of restricted families of projection, we refer to \cite{ChenP} and \cite{ChenRF}, respectively.


%
%
%
%


\medskip
\noindent{\bf Acknowledgements.} I am grateful to Tuomas Orponen for pointing out Theorem \ref{thm:new}.


\section{Preliminaries}\label{sec:p}

In this section we show some known lemmas for later use. The proofs of the following lemmas are based on the potential arguments.  For more details, we refer to  \cite[Chapter 6]{Falconer2003}, \cite[Chapter 9]{Mattila1995}, \cite[Chapter 5]{Mattila2015}. For Lemma \ref{lem:Mattilaf}, we provide an different  approach to \cite[Chapter 5]{Mattila2015}, and hence we show full details for it.

\begin{lemma}\label{lem:Mattila}
Let $G\subset G(n,m)$ and $\gamma$ be a  positive finite  Borel measure on $G$. If for any unit vector $\xi\in \RR^{n}$,
\begin{equation*}\label{eq:ll1}
 \gamma(\{V\in G: |\pi_{V}(\xi)|\leq \rho\})\lesssim \rho^{m},
\end{equation*} 
then  $((\pi_{V})_{V\in G}, \gamma)$ is a $\M$ space. 
\end{lemma}

\begin{lemma}\label{lem:Mattilaf}
Let $G\subset G(n,m)$ and $\gamma$ be a  positive finite  Borel measure on $G$. If for any unit vector $\xi\in \RR^{n}$,
\begin{equation*}
 \gamma(\{V\in G: d(\xi, V)\leq \rho\})\lesssim \rho^{n-m},
\end{equation*} 
then  $((\pi_{V})_{V\in G},  \gamma)$ is a  $\M$ space.
\end{lemma}

The proofs depend on the following energy characterization of Hausdorff dimension. For a Borel set $E\subset \RR^{n}$,
\begin{equation*}
\begin{aligned}
\dim_{H} E= \sup\{&s:  \I_{s}(\mu)<\infty, \\&\text{ $\mu$ is a nonzero Radon measure with compact support on $E$ }\}
\end{aligned}
\end{equation*}
where $\I_{s}(\mu)=\int\int |x-y|^{-s}d \mu x d \mu y$. We also need the following identity which connects the fractal geometry and Fourier analysis,
\[
\I_{s}(\mu)\approx \int_{\RR^{n}}|x|^{s-n}|\widehat{\mu }(x)|^{2}d x.
\]
Here $\widehat{\mu}(x)=\int e^{-2 \pi i \langle x, y \rangle}d \mu (y)$ the Fourier transform of the measure $\mu$ at $x$. For more connections between fractal geometry and Fourier analysis, we refer to \cite{Mattila2015}.

\begin{proof}[Proof of Lemma \ref{lem:Mattilaf}]
Let $\dim_{H}E=s\leq m$. Then for any $0<t<s$ there exists a Radon measure $\mu$  on $E$ with compact support and 
\begin{equation}\label{eq:condition}
\I_{t}(\mu)\approx \int_{\RR^{n}}|\widehat{\mu}(x)|^{2}|x|^{t-n}d  x d \gamma V <\infty.
\end{equation}
It is sufficient to prove 
\[
\int_{G}\int_{V} |\widehat{\mu_{V}}(x)|^{2}|x|^{t-m}d \h x d \gamma V <\infty.
\]
Note that for any $V\in G(n,m)$ and $x\in V$, 
\[
\widehat{\mu_{V}}(x)=\widehat{\mu}(x).
\]
Since the measure $\mu$ is finite, we have that there is a positive constant $C$ such that for any $V\in G(n,m)$,
\[
\int_{B(0,1)}|\widehat{\mu_{V}}(x)|^{2}|x|^{t-m}d \h x \leq C<\infty.
\]
Thus it is sufficient to prove 
\[
\int_{G}\int_{V\cap B(0,1)^{c}} |\widehat{\mu}(x)|^{2}|x|^{t-m}d \h x d \gamma V <\infty,
\]
where $B(0,1)^{c}$ is the complement set of $B(0,1)$.  Let $0<\rho<1/10\sqrt{n}$. Define 
\[
\mathcal{Q}_{\rho}:=\{[k_{1}\rho, (k_{1}+1)\rho]\times \cdots\times[k_{n}\rho, (k_{n}+1)\rho]: k_{j}\in \mathbb{Z}, 1\leq j\leq n\}=\{Q_{j}\}_{j=1}^{\infty}.
\]
We consider the cubes which intersects $B(0,1)^{c}$ only. Thus we define   
\[
J=\{j: Q_{j}\cap B(0,1)^{c}\neq \emptyset, Q_{j}\in \mathcal{Q}_{\rho}\}.
\] 
Since $\mu$ is a Radon measure with  compact support, $\widehat{\mu}$  is a bounded Lipschitz continuous function, i.e.,
\[
|\widehat{\mu}(x)-\widehat{\mu}(y)|\lesssim |x-y|.
\]
It follows that for  each $Q_{j}, j\in J$ and any  $x, x'\in Q_{j}$,
\begin{equation}\label{eq:uu}
|\widehat{\mu}(x)|^{2}|x|^{t-m}\lesssim |\widehat{\mu}(x')|^{2}|x'|^{t-m}+\rho |x'|^{t-m}.
\end{equation}
For each $Q_{j}, j\in J$ let $x_{j}\in Q_{j}$ such that 
\[
|\widehat{\mu}(x_{j})|\leq |\widehat{\mu}(x)|  \text{ for any } x \in Q_{j}.
\] 

For any $R>1$ let $\rho=\rho_{R}=R^{-m}$. Then the estimate \eqref{eq:uu} implies that for  any $V\in G(n,m)$, 
\begin{equation}\label{eq:f1}
\begin{aligned}
&\int_{V\cap B(0,1)^{c}\cap B(0,R)} |x|^{t-m}|\widehat{\mu }(x)|^{2} d \h x \\
&\lesssim  \sum_{j\in J}|\widehat{\mu }(x_{j})|^{2}|x_{j}|^{t-m} \mathcal{H}^{m}(V \cap Q_{j}\cap B(0,R))+\rho \h(V\cap B(0,R))\\
&\lesssim\sum_{j\in J}|\widehat{\mu }(x_{j})|^{2}|x_{j}|^{t-m} \mathcal{H}^{m}(V \cap Q_{j})+1.
\end{aligned}
\end{equation} 
For any $Q_{j}, j\in J,$  we have 
\begin{equation}\label{eq:f2}
\begin{aligned}
&\int_{G}\h(V\cap Q_{j})d \gamma V \\
&\lesssim \diam(Q_{j})^{m}\gamma(\{V\in G: d(x_{j}, V)\leq \diam(Q_{j})\})\\
&\lesssim \diam(Q_{j})^{m} \left(\frac{\diam (Q_{j})}{|x_{j}|}\right)^{n-m} \lesssim \diam(Q_{j})^{n}|x_{j}|^{m-n}.
\end{aligned}
\end{equation}
Combining with  Fatou's lemma, the estimates \eqref{eq:f1}, \eqref{eq:f2}, and the condition \eqref{eq:condition}, we obtain
\begin{equation*}
\begin{aligned}
&\int_{G}\int_{V\cap B(0,1)^{c}\cap B(0,R)} |x|^{t-m}|\widehat{\mu }(x)|^{2} d\h x d \gamma V\\
&\lesssim \int_{G}\sum_{j\in J} |x_{j}|^{t-m}|\widehat{\mu }(x_{j})|^{2}\h(Q_{j}\cap V) d \gamma V+1\\
&\lesssim \sum_{j\in J} |x_{j}|^{t-m}|\widehat{\mu }(x_{j})|^{2}\int_{G}\h(V\cap Q_{j})d \gamma V +1\\
&\lesssim \sum_{j\in J} |x_{j}|^{t-n}|\widehat{\mu }(x_{j})|^{2}\diam(Q_{j})^{n}+1\\
&\lesssim \I_{t}(\mu) +1<\infty. 
\end{aligned}
\end{equation*}
Together with Fatou's lemma, we obtain
\begin{equation*}
\begin{aligned}
&\int_{G}\int_{V\cap B(0,1)^{c}} |x|^{t-m}|\widehat{\mu }(x)|^{2} d\h x d \gamma V\\
&=\int_{G}\int_{V\cap B(0,1)^{c}} \lim_{R\rightarrow \infty}{\bf 1}_{B(0,R)}(x)|x|^{t-m}|\widehat{\mu }(x)|^{2} d\h x d \gamma V\\
&\leq \liminf_{R\rightarrow \infty} \int_{G}\int_{V\cap B(0,1)^{c}} {\bf 1}_{B(0,R)}(x)|x|^{t-m}|\widehat{\mu }(x)|^{2} d\h x d \gamma V\\
&\lesssim \I_{t}(\mu) +1<\infty. 
\end{aligned}
\end{equation*}
Thus we complete the proof of the dimension part of Marstrand-Mattila type theorem.

Now we turn to  the measure part of Marstrand-Mattila type theorem.  Let $\dim_{H}E=s>m$, then there exists a Radon measure $\mu$ on $E$ with compact support $\I_{m}(\mu)<\infty$. A variant of the former argument implies that (using the same notation as above)
\begin{equation*}
\begin{aligned}
&\int_{G}\int_{V\cap B(0,1)^{c}\cap B(0,R)} |\widehat{\mu }(x)|^{2} d\h x d \gamma V\\
&\lesssim \int_{G}\sum_{j\in J} |\widehat{\mu }(x_{j})|^{2}\h(Q_{j}\cap V) d \gamma V+1\\
&\lesssim \sum_{j\in J} |\widehat{\mu }(x_{j})|^{2}\int_{G}\h(V\cap Q_{j})d \gamma V +1\\
&\lesssim \sum_{j\in J}| \widehat{\mu }(x_{j})|^{2}|x_{j}|^{m-n}\diam(Q_{j})^{n}+1\\
&\lesssim \I_{m}(\mu) +1<\infty.
\end{aligned}
\end{equation*}
It follows that $\int_{G}\int_{V\cap B(0,1)^{c}} |\widehat{\mu_{V} }(x)|^{2} d\h x d \gamma V <\infty$. Recall that if $\int_{\RR^{n}}|\widehat{\mu}(x)|^2 d x<\infty $, then $\mu$ is absolutely continuous to $\mathcal{H}^{n}$ (see \cite[Theorem 3.3]{Mattila2015}). Thus  we obtain that $\mu_{V}$ is absolutely continuous to  $\h$, and hence $\h(\pi_{V}(E))>0$ for $\gamma$ almost all $V\in G$. 
\end{proof}

\section{Proofs of Theorems \ref{thm:main}-\ref{thm:new}}\label{sec:main}

P. Shmerkin and V. Suomala \cite{SS} constructed the following sets. A tube $T\subset \RR^{2}$ with width $\delta$ means that $T$ is a $\delta$ neighbourhood of some line in $\RR^{2}$. 

\begin{theorem}\label{lem:nicesets}
For any $\alpha\in (1, 2)$ there exists an $\alpha$-Ahlfors regular compact set $E\subset \RR^{2}$, such that for any tube $T$ with width $w(T)$,
\begin{equation}\label{eq:niceone}
\mathcal{H}^{\alpha}(E\cap T)\lesssim w(T).
\end{equation}
\end{theorem}

By ``mapping'' the sets in  Theorem \ref{lem:nicesets} to sphere $\mathbb{S}^{2}$, we obtain the following Lemma \ref{lem:ls}. 

\begin{lemma}\label{lem:ls}
For any $\alpha\in (1,2)$ there exists an $\alpha$-Ahlfors regular compact set $G\subset \mathbb{S}^{2}$, such that for  any unit vector $\xi \in \RR^{3}$ and $\rho>0$, 
\begin{equation}\label{eq:l1}
\mathcal{H}^{\alpha}(\{L\in G: |\pi_{L}(\xi)|\leq \rho\})\lesssim \rho.
\end{equation}
It follows that for any unit vector $\xi \in \RR^{3}$ and $\rho>0$,
\begin{equation}\label{eq:l2}
\mathcal{H}^{\alpha}(\{L\in G: d(\xi, L^{\perp})\leq \rho\})\lesssim \rho.
\end{equation}
\end{lemma}
\begin{proof}
By Theorem \ref{lem:nicesets} there exists an $\alpha$-Ahlfors regular compact set $E\subset [-1/10, 1/10]^{2}$
such that for any tube $T$,
\[
\mathcal{H}^{\alpha}(E\cap T)\lesssim w(T).
\]
Let $\tilde{E}=E+(0,0,1/2)$ and  $G:=\{\frac{x}{|x|}: x\in \tilde{E}\}$. We intend to prove that $G$ satisfy our need.

Note that  $G$ is the image of $\tilde{E}$ under the map $F: x \rightarrow \frac{x}{|x|}$ for $x\neq 0$. In the following, we restrict the map $F$ to the set $[-1/10, 1/10]^{2}+(0,0,1/2):=S$. Then $F$ is a bi-Lipschitz map, i.e.,
\[
|F(x)-F(y)|\approx |x-y|, \, \, x, y\in S.
\]
Furthermore, $F^{-1}$ map the ``big circle'' to some ``segment'' on $S$, i.e., for any plane $W\in G(3,2)$ there exists a line $\ell_{W}$ such that 
\[
F^{-1}(W\cap \mathbb{S}^{2})=\ell_{W}\cap S. 
\]
Combining with the bi-Lipschitz of the map $F$, we conclude that \[
F^{-1}(\{L\in G: |\pi_{L}(\xi)|\leq \rho\})\subset \{x\in \tilde{E} : d(x, \ell_{\xi^{\perp}})\lesssim \rho\}
\]  
where $\ell_{\xi^{\perp}}=F^{-1}(\xi^{\perp})$.
Therefore 
\[
\mathcal{H}^{\alpha}(\{L\in G: |\pi_{L}(\xi)|\leq \rho\})\approx \mathcal{H}^{\alpha} (F^{-1}(\{L\in G: |\pi_{L}(\xi)|\leq \rho\}))\lesssim \rho.
\]  
  
The estimate \eqref{eq:l2} follows by $d(\xi, L^{\perp})=\pi_{L}(\xi)$, thus we complete the proof.
\end{proof}

Then Theorem \ref{thm:main} follows by  combining  Lemma \ref{lem:ls}
 with Lemmas \ref{lem:Mattila}-\ref{lem:Mattilaf}. More precisely, the estimate \eqref{eq:l1} and the Lemma \ref{lem:Mattila} implies that $((\pi_{L_{x}})_{x\in G}, \mathcal{H}^{\alpha})$ is a $\M$ space. The estimate \eqref{eq:l2} and the lemma \ref{lem:Mattilaf} implies that $((\pi_{L_{x}^{\perp}})_{x\in G},  \mathcal{H}^{\alpha})$ is a  $\M$ space.

Now we turn to the proof of Theorem \ref{thm:new}. The method is similar to the proof of Theorem \ref{thm:main}. We  ``mapping'' some random sets  of plane to sphere $\mathbb{S}^{2}$, and then  applying the classical potential arguments for these restricted families of projections.   First note  that the classical potential arguments implies the following result, see the arguments in\cite[Section 3]{Falconer} or the proof in \cite[Theorem 5.1]{Mattila2015}.

\begin{lemma}\label{lem:low}
Let $G\subset G(n,m)$ and $\gamma$ be a  positive finite  Borel measure on $G$. If for any unit vector $\xi\in \RR^{n}$,
\[
 \gamma(\{V\in G: |\pi_{V}(\xi)|\leq \rho\})\lesssim \rho^{\alpha},
\] 
where $\alpha$ is a positive constant, then $((\pi_{V})_{V\in G}, \gamma)$ admit a (dimension part)  Marstrand-Mattila type  theorem i.e., for any subset $E\subset \RR^{n}$ with $\dim_{H}E\leq \min\{\alpha, m\}$, we have
\[
\dim_{H}\pi_{V}(E)=\dim_{H}E \text{ for }\gamma\,\, a.e. V\in G.
\]
\end{lemma}

T. Orponen \cite{OOO} constructed the following sets.  
\begin{theorem} 
For any $0<\alpha<1$ there exists a compact set $E\subset [0,1]^{2}$ with $0<\mathcal{H}^{\alpha}(E)<\infty$, such that 
such that for any tube $T$ with width $w(T)$,
\[
\mathcal{H}^{\alpha}(E\cap T)\lesssim w(T)^{\alpha}.
\]
\end{theorem}

Note that for any subset $E\subset \RR^{2}$ with $0<\mathcal{H}^{1}(E)<\infty$,
\[
\sup _{T} \frac{\mathcal{H}^{1}(E\cap T)}{w(T)}=\infty
\] 
where the supremum is over all tubes $T$ with width $w(T) > 0$. For more details, see \cite{OOO}. For the case $\alpha=1$, the author \cite{Chend} constructed the following set which settles a question of T. Orponen. There exists a compact set $E\subset [0,1]^{2}$ with $0<\mathcal{H}^{1}(E)<\infty$ such that for any $s<1$, and  for any tube $T$ with width $w(T)$,
\begin{equation}\label{eq:nicethree}
\mathcal{H}^{1}(E\cap T)\lesssim_{s} w(T)^{s}.
\end{equation}
Here $\lesssim_{s}$ means that the constant depends on $s$.

We mapping the above  sets to the sphere $\mathbb{S}^{2}$ in the same way as Lemma \ref{lem:ls}, and the similar arguments implies the following result. 
\begin{corollary}\label{cor:nn}
For any $0<\alpha\leq 1$ there exists a compact set $G\subset \mathbb{S}^{2}$ with $0<\mathcal{H}^{\alpha}(G)<\infty$, such that for any $s<\alpha$, and for any unit vector $\xi\in \RR^{3}$,
\[
\mathcal{H}^{\alpha}(\{L\in G: |\pi_{L}(\xi)|\leq \rho\})\lesssim_{s} \rho^{s}.
\] 
Note that for the case $0<\alpha<1$, we  have the following stronger estimate
\[
\mathcal{H}^{\alpha}(\{L\in G: |\pi_{L}(\xi)|\leq \rho\})\lesssim \rho^{\alpha}.
\]
\end{corollary}

Theorem \ref{thm:new} follows by combining Corollary \ref{cor:nn}  and Lemma \ref{lem:low}.

%
%
%
%
%
%
%
%
%
%
%
%

\end{document}